\documentclass[11pt]{article}
\usepackage{amsmath,amssymb,amsthm, epsfig}
\usepackage{amsfonts}
\usepackage{amsmath}
\usepackage{amssymb}
\usepackage{color}
\usepackage[mathscr]{eucal}

\title{An interplay between the weak form of Peano's theorem and structural aspects of Banach spaces}

\author{by  \vspace{0.3cm}  \\    \textsc{C. S. Barroso, \quad M. P. Rebou\c cas} \\ \textit{\footnotesize Universidade Federal do Cear\'a}  \\ \textit{\footnotesize Fortaleza, CE, Brazil}\vspace{0.1cm} \\ and \vspace{0.3cm} \\ \textsc{M. A. M. Marrocos} \\ \textit{\footnotesize Universidade Federal do Amazonas} \\ \textit{\footnotesize Manaus, AM, Brazil} }

\date{}

\newlength{\hchng}
\newlength{\vchng}
\def \supp {\mathrm{supp } }

\def\spann{\hbox{\tt span}\,}

\newtheorem{theorem}{Theorem}[section]
\newtheorem{lemma}[theorem]{Lemma}
\newtheorem{proposition}[theorem]{Proposition}

\theoremstyle{definition}
\newtheorem{definition}[theorem]{Definition}

\numberwithin{equation}{section}

\newcommand{\intav}[1]{\mathchoice {\mathop{\vrule width 6pt height 3 pt depth  -2.5pt
\kern -8pt \intop}\nolimits_{\kern -6pt#1}} {\mathop{\vrule width
5pt height 3  pt depth -2.6pt \kern -6pt \intop}\nolimits_{#1}}
{\mathop{\vrule width 5pt height 3 pt depth -2.6pt \kern -6pt
\intop}\nolimits_{#1}} {\mathop{\vrule width 5pt height 3 pt depth
-2.6pt \kern -6pt \intop}\nolimits_{#1}}}

\begin{document}
\maketitle

\begin{abstract}
In this paper we establish some new results concerning the Cauchy-Peano problem in Banach spaces. Firstly, we prove that if a Banach space $E$ admits a fundamental biorthogonal system, then there exists a continuous vector field $f\colon E\to E$ such that the autonomous differential equation $u'=f(u)$ has no solutions at any time. The proof relies on a key result asserting that every infinite-dimensional Fr\'echet space with a fundamental biorthogonal system possesses a nontrivial separable quotient. The later, is the byproduct of a mixture of known results on barrelledness and two fundamental results of Banach space theory (namely, a result of Pe{\l}czy\'nski on Banach spaces containing $L_1(\mu)$ and the $\ell_1$-theorem of Rosenthal). Next, we introduce a natural notion of weak-approximate solutions for the non-autonomous Cauchy-Peano problem in Banach spaces, and prove that a necessary and sufficient condition for the existence of such an approximation is the absence of $\ell_1$-isomorphs inside the underline space. We also study a kind of algebraic genericity for the Cauchy-Peano problem in spaces $E$ having complemented subspaces with unconditional Schauder basis. It is proved that if $\mathscr{K}(E)$ denotes the family of all continuous vector fields $f\colon E\to E$ for which $u'=f(u)$ has no solutions at any time, then $\mathscr{K}(E)\bigcup \{0\}$ is spaceable in sense that it contains a closed infinite dimensional subspace of $C(E)$, the locally convex space of all continuous vector fields on $E$ with the linear topology of uniform convergence on bounded sets.

\bigskip

\noindent \textit{MSC\,2010:} 34G20, 34A34, 46B20, 47H10, 34K07.

\medskip

\noindent \textbf{Keywords:} Peano's theorem, Separable Quotient Problem, Fr\'echet spaces, $\ell_1$-isomorphs, weak-compactness in $L_\infty(\mu,E)$, abstract fixed point approximation.

\end{abstract}

\newpage

\section{Introduction}
In \cite{Peano} Peano proved his famous theorem on the existence of solutions for first-order initial value problems in finite-dimensional spaces. The infinite-dimensional version of this result is not true in general. Indeed, the first negative result was given in 1950 by Dieudonn\'e \cite{Dieudonne}, who provided a counterexample in the space $c_0$. Along the years, several researchers have provided invaluable information for this fascinating subject (see, for instance, \cite{Astala, Cellina, Godunov, Lobanov,Szep} and references therein). In 1974 Godunov \cite{Godunov} settled the question concerning the non-validity of Peano's theorem in general Banach spaces. He proved that Peano's theorem is true in a Banach space $E$ iff it is finite dimensional. Generalizations of this result in the context of locally convex and non-normable Fr\'echet spaces were given by Astala \cite{Astala}, Lobanov \cite{Lobanov}, Shkarin \cite{Shkarin1} and others. Also, some existence results have been derived by using non-normable linear topologies (cf. \cite{Szep,Teixeira}). Nowadays, there is a new trend in this research area which aims at studying relationships between geometric and structural aspects of Banach spaces and the weak-form of Peano's theorem (cf. \cite{Hajek-Johanis,Shkarin2}). The weak form of Peano's theorem (WFPT, in short) ensures that if $E$ is finite dimensional and $f\colon \mathbb{R}\times E\to E$ is continuous, then $u'=f(t,u)$ has a solution in some open interval. In \cite{Shkarin2}, Shkarin proved that if $E$ is a Banach space with a complemented subspace having an unconditional Schauder basis, then the WFPT fails to be true. Haj\'ek and Johanis \cite{Hajek-Johanis} extended this result to the class of Banach spaces with an infinite-dimensional separable quotient. Actually, they proved the more stronger statement that there are continuous vector fields $f\colon E\to E$ such that the differential equation $u'=f(u)$ has no solutions at any point (cf. \cite[Theorem 8]{Hajek-Johanis}).

Henceforth $E$ will denote an infinite dimensional Banach space. The first goal of the present research is particularly focused on relationships between WFPT and fundamental biorthogonal systems. We obtain the following result.
\par
\begin{theorem}\label{thm:M1} Assume that $E$ admits a fundamental biorthogonal system. Then there exists a continuous mapping $f\colon E\to E$ such that the autonomous equation $u'=f(u)$ has no solutions at any time.
\end{theorem}

The proof is based on an apparently new contribution for the Separable Quotient Problem (SQP, in short), see Theorem \ref{thm:M13}, and a result of Haj\'ek and Johanis \cite[Theorem 8]{Hajek-Johanis} which connects SQP to the validity of the WFPT in Banach spaces.

\par
In our second aim, we consider the non-autonomous  Cauchy-Peano problem
\begin{equation}\label{eqn:1}
u'=f(t,u),\quad t\in I\quad \text{ and }\quad u(t_0)=u_0\in E
\end{equation}
where $I=[0,T]$. Here $f\colon [0,T]\times E\to E$ is a Caracth\'eodory vector field, i.e., one with the following properties:
\begin{itemize}
\item[$(f_1)$] for all $t\in I$, $f(t,\cdot)\colon E\to E$ is continuous,
\item[$(f_2)$] for all $x\in E$, $f(\cdot,x)\colon E\to E$ is measurable.
\end{itemize}

We identify a structural condition which characterizes the existence of weak-approximate solutions (WAS, in short) to the problem (\ref{eqn:1}), see Definition \ref{def:WAS} for WAS. More precisely, for a large class of Caracth\'eodory vector fields, we obtain a characterization of the existence of WAS in terms of $\ell_1$-containements. In the explanation below, we describe the class of fields that will be addressed here. Let $I\subset \mathbb{R}$ be as above and denote by $\mathfrak{X}(I,E)$ the family of all Caracth\'eodory vector fields $f\colon I\times E\to E$ fulfilling the following growth condition:\vskip .1cm
\begin{itemize}
\item[$(\star)$] $\| f(s,x)\|_E\leq \alpha(s)\varphi(\|x\|_E)$ for a.e. $s\in I$ and every $x\in E$,
where $\alpha$ and $\varphi$ satisfy the following properties: (a) $\alpha\in L_1[0,T]$; and (b) $\varphi\colon [0,\infty)\to (0,\infty)$ is a nondecreasing continuous function satisfying
\[
\int_0^T \alpha(s)ds<\int_0^\infty \frac{ds}{\varphi(s)}.
\]
\end{itemize}

\begin{definition}[WAS]\label{def:WAS} Let $f\in \mathfrak{X}(I,E)$ be given. A sequence $(u_n)\subset C(I,E)$ of continuous $E$-valued functions on $I$ is said to be a weak-approximate solution of (\ref{eqn:1}) if:
\begin{itemize}
\item[(i)] Each $u_n$ is almost everywhere strongly differentiable in $I$,
\item[(ii)] Both $(u_n)$ and $(u'_n)$ are bounded sequences in $C(I,E)$,
\item[(iii)] $(u_n)$ is weakly Cauchy and satisfies $\displaystyle u_n-\int f(s,u_n(s))ds \rightharpoonup u_0$ in  $C(I,E)$,
\item[(iv)] $u_n(t)-u_0\in \overline{\spann}\left( f(I\times E)\right)$ for all $t\in I$ and $n\in\mathbb{N}$.
\end{itemize}
\end{definition}

\par

Our second main result is as follows.
\begin{theorem}\label{thm:M2} Problem (\ref{eqn:1}) always have WAS for $f\in \mathfrak{X}(I,E)$ if and only if $E$ contains no subspace isomorphic to $\ell_1$.
\end{theorem}

This is a kind of generalization of Theorem 4.2 in \cite{Barroso}. The proof given here relies on three important results from the functional analysis: (1) A fundamental characterization of weak compactness in $L_\infty(\mu,E)$ due to Schl\"uchtermann \cite{Schluchtermann}; (2) A characterization of reflexiveness due to Cellina \cite{Cellina}, and (3) the famous Rosenthal's $\ell_1$-theorem.

The third goal of this paper is concerned with the analysis of a kind of genericity of solutions for abstract differential equations. Let $C(E)$ be the locally convex space of all continuous vector fields on $E$, endowed with the linear topology $\mathscr{T}_{uc}$ of uniform convergence on bounded sets. Denote by $\mathscr{K}(E)$ the family of all vector fields in $C(E)$ for which (\ref{eqn:1}) does not have solutions at any time. The central question we address here is: Does $\mathscr{K}(E)\bigcup\{0\}$ contains an infinite dimensional $\mathscr{T}_{uc}$-closed vector space? This suggests the following definition of algebraic genericity.

\begin{definition} We say a property $(P)$ is algebraically generic for $\mathscr{K}(E)$ if $\mathscr{K}(E)\bigcup\{0\}$ contains an infinite dimensional $\mathscr{T}_{uc}$-closed vector space $L$ such that it holds for all the non-zero vector fields in $L$.
\end{definition}

In this direction, we obtain the following result.

\begin{theorem}\label{thm:M3} Assume that $E$ contains a complemented subspace with an unconditional Schauder basis. Then the property of the non-validity of the WFPT is algebraically generic for $\mathscr{K}(E)$.
\end{theorem}

The source of inspiration for this result is an idea present, for instance, in the recent papers \cite{Aron1,Aron2,Barroso2}, which suggest the usefulness in studying lineability and spaceability issues. In particular, Theorem \ref{thm:M3} generalizes \cite[Theorem 2.1]{Barroso2}. The proof borrows some ideas from \cite{Barroso2}, but the details are  realized via a different and rather simple way.  Indeed, the proof consists in obtaining a countable family of linearly independent vector fields $\{f_n\}$ on $E$ for which the system of ODEs $u'=f_n(u)$ treated as an uncoupled system, has no solution at any time in $E$. To this intended, a special role is played by a result of Shkarin \cite{Shkarin2} concerning Osgood's Theorem in Banach spaces having complemented subspaces with unconditional basis. At this point, it is worth to point out the uniform continuity of the family $\{f_n\}$ which is a crucial step of the proof. This is thanks to the aforementioned result.

\subsection{Preliminary notation and definitions} \label{subsec:1.1}\rm
Before starting the proofs, we recall some basic notation. All of the Banach spaces we consider are over the reals. Given a Banach space $X$ with norm $\|\cdot\|_X$, we will denote the closed unit ball by $B_X$. If $A\subset X$ is nonempty, we will denote the linear span of $A$ by $\spann(A)$, and the closed linear span by $\overline{\spann}(A)$. Subspaces of Banach spaces are understood to be closed infinite dimensional subspaces. $(e_i)$ stands for the usual unit ball of $\ell_1$, the vector space of all summable sequences of real numbers. As usual, we also write $X\approx Y$ to denote that the spaces are isomorphic. A Banach space $X$ is said to contain an isomorphic copy of $\ell_1$ if there is a basic sequence $(x_n)\subset X$ such that $\{x_n\}$ and $\{e_i\}$ are equivalent. This in turn is equivalent to the existence of constants $C_1,C_2>0$ such that for all scalars $t_1,t_2,\dots, t_n$
\[
C_1^{-1} \sum_{i=1}^n|t_i|\leq \left\| \sum_{i=1}^n t_i x_i\right\|_X\leq C_2 \sum_{i=1}^n|t_i|.
\]
Recall that a series $\sum x_n$ in a Banach space $X$ is unconditionally convergent if $\sum\epsilon_n x_n$ converges for all choices of signs $\epsilon_n=\pm 1$. An a basic sequence $(x_n)$ is said to be unconditional if for every $x\in\overline{\rm{span}}(\{x_n\})$, its expansion $x=\sum a_n x_n$ converges unconditionally. The reader is referred to \cite{Fabian} for more background in Banach space theory.

\section{The Separable Quotient Problem and the weak form of Peano's theorem}
\hskip .5cm
The Separable Quotient Problem asks whether every infinite dimensional Banach space $E$ has a non-trivial separable quotient, i.e., a closed infinite-dimensional subspace $M$ so that $E/M$ is linearly isometric to a separable infinite dimensional Banach space. This still open problem is essentially due to Banach and Pe{\l}czy\'nski. Many special spaces (e.g. if $E$ is either separable or reflexive) are known to have nontrivial separable quotients. The works of Johnson-Rosenthal \cite{JR}, Hagler-Johnson \cite{HJ} and Argyros-Dodos-Kanellopoulos \cite{ADK} are certainly among the most important contributions for this problem. We would like also to refer the reader to the survey article of Mujica \cite{Mujica}, the recent work of \'{S}liwa \cite{Sliwa} and references therein for other important contributions. In this section we get the following result, which seems to be an apparently new progress for the (SQP). It also leads us to a proof of Theorem \ref{thm:M1}.
\begin{theorem}\label{thm:M13} Every infinite dimensional Fr\'echet space with a fundamental biorthogonal system has a non-trivial separable quotient.
\end{theorem}

\begin{proof}
Let $(E,\tau)$ be a Fr\'echet space with a fundamental biorthogonal system $\{ e_\gamma; e^*_\gamma\}_{\gamma\in\Gamma}$. If $\tau$ is non-normable then in this case, by Theorem 2.6.16 in \cite{Carreras-Bonet}, $(E,\tau)$ has a quotient isomorphic to $\omega$. Thus we may assume the existence of a continuous norm $\|\cdot\|$ giving the topology $\tau$. Without loss of generality, we can also assume that $\|e_\gamma\|= 1$ for all $\gamma\in \Gamma$. Let us suppose the contrary, that $E$ contains no non-trivial infinite-dimensional separable quotient. According to Proposition 4.6.5 of \cite{Carreras-Bonet}, this means that every proper dense subspace of $E$ is barrelled. Thus the linear space $F\colon={\rm{span}}\{e_\gamma\}_{\gamma\in \Gamma}$ is barrelled. A direct computation shows that  $\{e_\gamma\}_{\gamma\in \Gamma}$ is a Hamel Schauder basis of $F$ and $\{e^*_\gamma\}_{\gamma\in\Gamma}$ its corresponding coefficient functionals. Observe also that if $x=\sum_{\gamma\in\Gamma} e^*_\gamma(x)e_\gamma\in F$, then $\|x\|\leq \sum_{\gamma\in\Gamma} |e^*_\gamma(x)|$. Thus, we can then define a norm $\||\cdot|\|$ on $F$ by putting
\[
\||x|\|=\sum_{\gamma\in\Gamma} |e^*_\gamma(x)|,\quad  x=\sum_{\gamma\in\Gamma} e^*_\gamma(x)e_\gamma\in F.
\]

\par
\paragraph{Claim.} $\|\cdot\|$ and $\||\cdot|\|$ are equivalent on $F$. It suffices to show that for some constant $C>0$, $\||x|\|\leq C \|x\|$ for all $x\in F$. We will prove this in two steps.

\par

\paragraph{\it Step one.} The linear mapping $S\colon (F,\|\cdot\|)\to (\ell_1(\Gamma),\|\cdot\|_{\ell_1(\Gamma)})$ given by
$$
S(x)=(e^*_\gamma(x))_\gamma,\quad x\in F
$$
is continuous. Indeed, since $F$ is barrelled, it is enough to show that $S$ has closed graph. Assume then that $u_k\to u$ in $(F,\|\cdot\|)$ and $S(u_k)\to v$ in $(\ell_1(\Gamma),\|\cdot\|_{\ell_1(\Gamma)})$. Let $\varepsilon>0$ be arbitrary. Write $v=(v_\gamma)_{\gamma\in \Gamma}$. Then for all sufficiently large $k$, we get that
\[
\sum_{\gamma\in\Gamma}|e^*_\gamma(u_k)-v_\gamma|<\varepsilon,
\]
and hence we have
\[
|v_\gamma|<\varepsilon + |e^*_\gamma(u_k)|,\quad\forall \gamma\in\Gamma.
\]
Passing to the limit as $k\to\infty$, it follows
\[
|v_\gamma|\leq \varepsilon+|e_\gamma^*(u)|.
\]
Since $\varepsilon>0$ was arbitrary, this implies that $|v_\gamma|\leq |e_\gamma^*(u)|$ for all $\gamma\in\Gamma$. Let $\supp(u)$ denote the support of $u$. Thus if $\gamma\not\in \supp(u)$, then $e_\gamma^*(u)=0$ and so $v_\gamma=0$. Consequently, $\supp(v)$ is finite. On the other hand, the inequality
\[
|e_\alpha^*(u_k) - v_\alpha|\leq \sum_{\gamma\in\Gamma} |e_\gamma^*(u_k) - v_\gamma|= \|S(u_k) - v\|_{\ell_1(\Gamma)}
\]
implies that $e_\alpha^*(u)=v_\alpha$ for all $\alpha\in\Gamma$. Thus,
\[
S(u)=(e_\gamma^*(u))_\gamma=(v_\gamma)_\gamma=v.
\]
This shows that $S$ has closed graph in $(F,\|\cdot\|)\times (\ell_1(\Gamma),\|\cdot\|_{\ell_1(\Gamma)})$ and, hence, is continuous.

\par

\paragraph{\it Step two.} The unit ball $B=\{ x\in F\colon \||x|\|\leq 1\}$ of $(F,\||\cdot|\|)$ is convex, balanced, absorbing and $\|\cdot\|$-closed in $F$. It suffices to prove that it is closed in $(F,\|\cdot\|)$. To prove this, suppose that $(u_k)$ is a sequence in $B$ so that $\|u_k - u\|\to 0$ for some $u\in F$. Since $S$ is continuous, we have
$$
\|S(u_k) - S(u)\|_{\ell_1(\Gamma)}\to 0.
$$
This and the inequality
\[
\||u|\|=\sum_{\gamma\in\Gamma}|e_\gamma^*(u)|=\|S(u)\|_{\ell_1(\Gamma)}\leq \|S(u_k) - S(u)\|_{\ell_1(\Gamma)}+1
\]
imply that $u\in B$.

\par

Now using again the fact that $(F,\|\cdot\|)$ is barrelled, we conclude that $B$ is a neighborhood of $0$ in $(F,\|\cdot\|)$. In particular, this implies the existence of a constant $C>0$ so that $\||x|\|\leq C\|x\|$ for all $x\in F$. Therefore, $(F,\|\cdot\|)$ is isomorphic to $(F,\||\cdot|\|)$. On the other hand, notice that $(F,\||\cdot|\|)$ clearly has Schur's property. Then $(F,\|\cdot\|)$ has itself the same property. The denseness of $F$ in $(X,\|\cdot\|)$ and Hahn-Banach theorem imply that $(X,\|\cdot\|)$ has also Schur's property. Thus, in view of the classical Rosenthal's $\ell_1$-theorem (see \cite{Rosenthal}), $(X,\|\cdot\|)$ contains an isomorphic copy of $\ell_1$. If $\ell_1$ embeds into $X$, then by a result of Pe{\l}czy\'nski \cite{Pe} we have that $L^1[0,1]$ embeds into $X^*$. Hence $\ell_2$ embeds into $X^*$ which implies that $\ell_2$ is a quotient of $X$. This contradiction concludes the proof of theorem.
\end{proof}

\par

\subsection{Proof of Theorem \ref{thm:M1}}

The proof follows immediately from Theorem 8 in \cite{Hajek-Johanis} and Theorem \ref{thm:M13}.\hfill $\square$

\section{Characterization of Banach spaces containing $\ell_1$ in terms of WAS for the Cauchy-Peano problem}
The celebrated dichotomy theorem of Rosenthal \cite{Rosenthal} states that every bounded sequence in a Banach space $E$ either has a weak Cauchy subsequence or admits a subsequence which is equivalent to the unit vector basis of $\ell_1$. In this section, we shall use again this result together a characterization of weak-compactness in $L_\infty(\mu,E)$, due to Schl\"uchtermann \cite{{Schluchtermann}}, as tools to prove Theorem \ref{thm:M2}.
\subsection{Proof of Theorem \ref{thm:M2}}
\noindent{\it Necessity.} Suppose $X$ is a subspace of $E$ which is isomorphic to $\ell_1$. Then $X$ is not reflexive, and by a result of Cellina \cite{Cellina}, there exist a continuous linear functional $\vartheta\in B_{X^*}$ with $\|\vartheta\|=1$, and a continuous map $g\colon B_{X}\to B_{X}$ which is fixed-point free and satisfies for any $x\in B_{X}$, the equality $\langle\vartheta, g(x)\rangle=\frac{1}{2}( \langle \vartheta, x\rangle +1)$. Let $G\colon E\to B_{X}$ be a continuous extension of $g$ to the whole $E$, with range in $B_{X}$. Following Cellina \cite{Cellina}, we define a continuous vector field $f_G\colon \mathbb{R}\times E\to E$ by
\begin{eqnarray*}
f_G(t,x)=\left\{
\begin{array}{lll}
2t G(x/t^2),\quad &&t\neq 0,\\
0,\quad &&t=0.
\end{array}\right.
\end{eqnarray*}
Notice that
\begin{equation}\label{eqn:est1LCDTh}
\|f_G(t,x)\|_{E}\leq 2|t|,\quad \forall t\in\mathbb{R},\,\,\forall x\in E.
\end{equation}
Thus $f_G\in \mathfrak{X}(\mathbb{R},E)$ with $\alpha(t)=2|t|$ and $\varphi\equiv 1$. In \cite{Cellina}, Cellina proved that there is no solution for the Cauchy-problem
\begin{eqnarray}\label{eqn:Cellina1}
\left\{
\begin{array}{lll}
u'(t)=f_G(t,u(t)),\\
u(0)=0.
\end{array}\right.
\end{eqnarray}
This is equivalent to the statement that the integral equation
\begin{eqnarray}\label{eqn:Cellina2}
u(t)= \int_0^t f_G(s,u(s))ds
\end{eqnarray}
does not have solutions.

\par

\paragraph{\it Claim 1.} (\ref{eqn:Cellina1}) does not have WAS. Indeed, suppose to the contrary that (\ref{eqn:Cellina1}) admits an WAS. Then for some bounded interval $I\subset \mathbb{R}$, containing $0$, there is a sequence $(u_n)\subset C(I,E)$ satisfying conditions (i)--(iv) from Definition \ref{def:WAS}. First of all, since $f_G(I\times E)\subset X$, we see from Definition \ref{def:WAS}-(v) that $u_n(t)\in X$ for all $n\in\mathbb{N}$ and $t\in I$. Furthermore, by Definition \ref{def:WAS}-(iii), it follows that
\[
u_n(t)-\int_0^t f_G(s,u_n(s))ds\rightharpoonup 0 \text{ in } X, \quad\forall t\in I.
\]
Now as $X\approx \ell_1$, $X$ shares the same properties of $\ell_1$. In particular, $X$ has Schur's property and is $\sigma(X,X^*)$-sequentially complete. Hence for each $t$,
\begin{eqnarray}\label{eqn:Cellina2}
u_n(t)-\int_0^tf_G(s,u_n(s))ds\to 0 \textrm{ in } X.
\end{eqnarray}
On the other hand, it is easily seen from Definition \ref{def:WAS}-(iv) that each sequence $(u_n(t))_n$ is weakly-Cauchy in $X$. Let $t\in I$ be fixed. As $X$ is $\sigma(X,X^*)$-sequentially complete, it follows that $(u_n(t))$ converges weakly (and so strongly) to some $u(t)\in X$. Using that estimative in (\ref{eqn:est1LCDTh}) and applying Lebesgue Dominated Convergence Theorem, (\ref{eqn:Cellina2}) implies $u(t)=\int_0^t f_G(s,u(s))ds$, $\forall t\in I$. Thus $u$ belongs to $C(I,E)$ and is a solution of (\ref{eqn:Cellina2}). This gives a contradiction to the fact that (\ref{eqn:Cellina1}) does not admit solutions, and completes the proof of the necessity.

\medskip

\noindent{\it Sufficiency.} Informally, the strategy is to show that the mapping $F\colon C(I,E)\to C(I,E)$ given by
\[
F(u)(t)=u_0+\int_0^t f(s,u(s))ds,\quad t\in I
\]
has a weak-approximate fixed point sequence, that is, a sequence $(u_n)$ so that $u_n - F(u_n)\rightharpoonup 0$ in $C(I,E)$. Such a sequence will be a WAS of (\ref{eqn:1}). To this end, let us consider the sets $A, B$ and $C$ defined below:
\medskip
\begin{eqnarray*}
\begin{array}{lll}
&&A=\left\{ u\in C(I,E)\colon \|u(t)\|_E\leq b(t) \textrm{ for a.e. } t\in I\right\},
\\ \\
&&B=\left\{ v\in C(I,E)\colon v(I)\subset W,\, \|v(t)\|_E\leq \alpha(t) \varphi(b(t))\textrm{ for a.e. } t\in I\right\},
\\ \\
&&C=\left\{ u\in A\colon u(t)=u_0 +\displaystyle\int_0^t \hat{u}(s)ds, \textrm{ for a.e. } t\in I \textrm{ and some } \hat{u}\in B\right\},
\end{array}
\end{eqnarray*}
where $W=\overline{\spann}(f(I\times E))$. Here $b\colon [0,\infty)\to \mathbb{R}$ is defined by $b(t)=J^{-1}\big(\int_0^t \alpha(s)ds\big)$,
where $J(z)=\int_{\|u_0\|_E}^z \frac{ds}{\varphi(s)}$. Clearly both $A$ and $B$ are closed convex subsets of $C(I,E)$. While $C$ is bounded and convex. Moreover, it is easy to see that  $F(C)\subset C$. The remainder of the proof is composed of three steps.

\par

\paragraph{\it Step A.} $F$ is demicontinuous. Indeed, suppose that $u_n\to u$ in $C$. We must to show that $F(u_n)\rightharpoonup F(u)$ in $C(I,E)$. By assumption, $u_n(t)\to u(t)$ in $E$ for all $t\in I$. Since $f$ is Carath\'eodory, this implies $f(t,u_n(t))\to f(t,u(t))$ in $E$ for a.e. $t\in I$. On the other hand, condition ($f_2$) shows that $\|f(t,u_n(t))\|_E\leq b'(t)$ for all $t\in I$. So by Lebesgue's Dominated Convergence Theorem, we get
\[
\int_0^t \|f(s,u_n(s)) - f(s,u(s))\|_Eds \to 0, \forall t\in I.
\]
In particular, we have $F(u_n)(t)\to F(u)(t)$ for all $t\in I$. The next claim is a key point to conclude the demicontinuousness of $F$.\vskip .2cm

\paragraph{\it Claim 2.} $K=\left\{ F(u_n)\colon n\in\mathbb{N}\right\}$ is relatively weakly compact. In order to prove this, we need the following characterization of weak compactness in $L_\infty(I,E)$ (see \cite[Theorem 2.7]{Schluchtermann} for its full statement):

\begin{theorem}\label{thm:WC} Let $(\Omega, \Sigma,\mu)$ be a positive and finite measure space. For a bounded subset $K\subset L_\infty(\mu,E)$ the following are equivalent:
\begin{itemize}
\item[(a)] $K$ is relatively weakly compact.
\item[(b)] For any sequence $(v_n)\subset K$ there exist a subsequence $(w_i)$ of $(v_n)$, a function $w\in L_\infty(\mu,E)$ and a set $N\subset \Omega$ with $\mu(N)=0$, such that:
    \begin{itemize}
    \item[(i)] $\forall t\in \Omega\setminus N$, $w_i(t)\to w(t)$ weakly in $E$,
    \item[(ii)] $\forall\,(x_j^*)\subset B_{E^*}$, $\forall\, (t_j)\subset \Omega\setminus N$, there exist subsequences $(x^*_{j_k})$,\, $(t_{j_k})$ such that
        \[
        \lim_{i\to\infty} \lim_{k\to\infty}\langle x^*_{j_k}, \left( w_i - w\right)(t_{j_k})\rangle=0.
        \]
    \end{itemize}
\end{itemize}
\end{theorem}
\noindent Let $(v_i)$ be any subsequence of $\left( F(u_n)\right)$ and $(w_i)$ any subsequence of $(v_i)$. Set $w=F(u)$, and pick any set $N\subset I$ with $|N|=0$. We already have proved that $w_i(t)\to w(t)$, for all $t\in I\setminus N$. Hence, condition (i) in Theorem \ref{thm:WC} is fulfilled.
Let now $(x^*_j)$ be any sequence in the unit ball $B_{E^*}$ of $E^*$, and let $(t_j)$ be any sequence of real numbers in $I\setminus N$. It is easy to check that
\[
\left|\langle x^*_j , \left( w_i - w\right)(t_j)\rangle\right|\leq \int_0^T \|f(s,u_{m_i}(s)) - f(s,u(s))\|_Eds
\]
where we are assuming that $w_i=F(u_{m_i})$ for all $i\in\mathbb{N}$. Using again Lebesgue's Dominated Convergence Theorem, it is easy to verify that
\[
\lim_{i\to\infty}\int_0^T \left\| f(s,u_{m_i}(s)) - f(s,u(s))\right\|_E ds=0
\]
which implies that
\[
\lim_{i\to\infty}\lim_{j\to\infty}\left| \langle x^*_j, \left( w_i - w\right)(t_j)\rangle\right|=0.
\]
Thus we get for free the assumption (ii) in Theorem \ref{thm:WC}, given the arbitrariness of the sequences $(x^*_j)$ and $(t_j)$. Hence $K$ is relatively weakly compact in $L_\infty(I,E)$, and the proof of Claim 2 is complete.

As is well known, Claim 2 implies that for some subsequence $(u_{n_k})$ of $(u_n)$ the sequence $(F(u_{n_k}))$ converges weakly to some $v$ in $L_\infty(I,E)$. Then, by Theorem 2.11 in \cite{Teixeira}, for almost every $t\in I$ we have that $F(u_{n_k})(t)\rightharpoonup v(t)$ in $E$. Since the weak topology is Hausdorff, it follows that $v\equiv F(u)$ a.e. in $I$. It is not hard to conclude that $F(u_n)$ converges weakly to $F(u)$ in $C(I,E)$. This proves that $F$ is demicontinuous.

\par

\paragraph{\it Step B.} $F$ has a weak-approximate fixed point sequence in $C$. To prove this, we need the following result as a crucial tool.

\begin{lemma}\label{lem:BKL} Let $X$ be a Banach space, $C\subset X$ a bounded convex set and $F\colon C\to \overline{C}$ a demicontinuous map. Assume that $C$ does not contain any isomorphic copy of $\ell_1$. Then there exists a sequence $(u_n)$ in $C$ so that $u_n - F(u_n)\rightharpoonup 0$ in $X$.
\end{lemma}

\begin{proof} This is a direct consequence of Proposition 3.7 in \cite{Barroso-Kalenda-Lin}.
\end{proof}

Let us turn out the attention to the proof of Theorem \ref{thm:M1}. Since $I$ is not scattered and $X$ contains no isomorphic copy of $\ell_1$, by a result of Cembranos \cite{Cembranos}, it follows that $C(I,X)$ also cannot have any isomorphic copy of $\ell_1$. Thus, by Lemma \ref{lem:BKL}, there is a sequence $(u_n)$ in $C$ so that $u_n - F(u_n)\rightharpoonup 0$ in $C(I,E)$. This concludes the proof of Step B.

\vskip .2cm
\paragraph{\it Step C.} The sequence $(u_n)$ gained in {\it Step B} is a WAS for problem (\ref{eqn:1}). Indeed, conditions (i)-(ii) and (iv) of the Definition \ref{def:WAS} follow easily from the fact that the sequence $(u_n)$ belongs to $C$. Finally, after passing to a subsequence if needed, item (iii) is a direct consequence of Rosenthal's $\ell_1$-theorem \cite{Rosenthal} and the fact that $u_n - F(u_n)\rightharpoonup 0$ in $C(I,E)$. The proof of Step C is complete, and the proof of Theorem \ref{thm:M1} is concluded. \hfill $\square$

\section{The algebraic genericity for the WFPT in Banach spaces}
In this section we establish the algebraic genericity of the weak form of Peano's theorem in Banach spaces having complemented subspaces with unconditional Schauder basis.
\subsection{Proof of Theorem \ref{thm:M3}}
According to the assumption, there exists a complemented subspace $X$ of $E$ with an unconditional Schauder basis $\{e_n;e^*_n\}_{n=1}^\infty$. As $X$ is complemented, there exists a bounded linear projection $P$ of $E$ onto $X$ (cf. \cite{Fabian}).  Split $\mathbb{N}$ into $\mathbb{N}=\bigcup_{i\geq 1}\mathbb{N}_{i}$, where each $\mathbb{N}_i$ has cardinality $|\mathbb{N}_{i}|=\infty$ and $\mathbb{N}_{i}\cap \mathbb{N}_{j}=\emptyset$ if $i\neq j$. We shall use the convention $\mathbb{N}_0=\mathbb{N}$. Let $X_i=\overline{\spann}\{e_n\colon n\in\mathbb{N}_{i}\}$. We then define, for each $i\in\mathbb{N}$, the $i$-th projection $\pi_i$ from $X$ into $X_i$ by
\begin{equation}\label{eqn:iprojection}
\pi_i(x)=\sum_{n=1}^\infty \big( e^*_{\mathbb{N}_{i}}\big)_n(x)e_n,\quad x\in X
\end{equation}
where $\big( e^*_{\mathbb{N}_{i}}\big)_n(x)=e_n^*(x)$ if $n\in\mathbb{N}_{i}$ and $0$ else. Since $\{ e_n\colon n\in\mathbb{N}_i\}$ is an unconditional Schauder basis for $X_i$ (see \cite[Proposition 6.31]{Fabian}), we have that $\pi_i(x)$ is well-defined and $\sup_{i\in\mathbb{N}}\|\pi_i\|<\infty$. Moreover, notice that $X_i=\pi_i(X)$. Now a crucial ingredient which will be used below is contained in the following lemma, whose proof can be found in \cite[Corollary 1.5]{Shkarin2}.

\begin{lemma}\label{lem:Shkarin} Let $X$ be a Banach space with a complemented subspace which has an unconditional basis. Then for any $\alpha\in (0,1)$ and $\epsilon>0$, there exists $f\colon X\to X$ such that
\begin{itemize}
\item[(i)] $\|f(x) - f(y)\|\leq \epsilon \|x - y\|^\alpha$ for all $x,y\in X$.
\item[(ii)] equation $u'=f(u)$ has no solutions in any interval of the real line.
\end{itemize}
\end{lemma}

Fixe $\alpha\in (0,1)$. According to Lemma \ref{lem:Shkarin}, for each $i\in\mathbb{N}$ there exists a vector field $f_i\colon X_i\to X_i$ satisfying properties $(i)$ and $(ii)$ above. Let $h_i\colon E\to E$ be a continuous vector field given by the formula
\[
h_i(x)=f_i\big( \pi_i(Px) \big),\quad x\in E.
\]
Now for each $(a_n)\in \ell_1$ we define another vector field $f_{(a_n)}$ on $E$ by putting
\[
f_{(a_n)}(x)=\sum_{i=1}^\infty a_i h_i(x),\quad x\in E.
\]
A direct computation shows that $f_{(a_n)}\equiv 0$ iff $(a_n)=0$ and
\[
\| f_{(a_n)}(x) - f_{(a_n)}(y)\|\leq \sup_{i\in\mathbb{N}}\|\pi_i\|\|P\| \| (a_n)\|_{\ell_1} \|x - y\|^\alpha,\quad \forall x,y\in E.
\]
With the same reasoning we can prove also the following proposition.
\begin{proposition} The linear map $T\colon \ell_1 \to C(E)$ given by
\[
T( (a_n)) = f_{(a_n)}
\]
is well-defined and is injective and continuous.
\end{proposition}
We can conclude therefore that $T(\ell_1)$ is algebraically isomorphic to $\ell_1$. Now we claim that
\begin{equation}\label{eqn:lineability}
T(\ell_1)\subset \mathscr{K}(E)\cup \{0\}.
\end{equation}
Indeed, let $(a_n)\in\ell_1\setminus \{0\}$ be arbitrary. Then $a_m\neq 0$ for some integer $m\geq 1$. Assume the contrary that $T( (a_n) )\not\in \mathscr{K}(E)$. Then there exists some open interval $I\subset \mathbb{R}$ so that the ODE
\begin{equation}\label{Eqn:US}
u'(t) = f_{(a_n)}(u(t))
\end{equation}
has a solution on $I$, say $u$. Let us define a vector-valued function $v\colon I\to E_m$ by $v(t)=(\pi_m\circ P)( u(t/a_m))$. Taking into account that (\ref{Eqn:US}) is an uncoupled system of infinite ODEs, after projecting and making the calculations of derivatives, we get
\[
v'(t)= f_m(v(t))\quad\forall t\in I
\]
which is a contradiction with the fact that for the field $f_m$ the WFPT fails in $E_m$. This concludes the proof of inclusion (\ref{eqn:lineability}). It remains to prove the following more refined inclusion
\[
\overline{T(\ell_1)}^{\mathscr{T}_{uc}}\subset\mathscr{K}(E)\cup \{0\}.
\]
To see this, let $h\in \overline{T(\ell_1)}^{\mathscr{T}_{uc}}$ be arbitrary. Then there is a sequence $x_k=(a_n^k)_{n=1}^\infty\in\ell_1$ so that $\{ \sum_{n=1}^\infty a_n^k f_n(\pi_n(x))\}_{k=1}^\infty$ converges uniformly on $B_E$ to $h$. Thus it is also Cauchy in $E$, uniformly in $x\in B_E$. Since each projection $\pi_j$ maps Cauchy sequences into Cauchy sequences, $(x_k)$ is Cauchy in $\ell_1$. Therefore we can pick $x=(a_n)\in \ell_1$ so that $\| x_k - x\|_{\ell_1}\to 0$, and as $\mathscr{T}_{uc}$ is Hausdorff we have $T(x)=h$. This completes the proof of theorem.
\hfill $\square$

\bigskip
\noindent{\bf Acknowledgment.} The research described in this paper was started while the first author was visiting Universidade Federal do Amazonas, during the period August-November, 2011, and finished during the 17th Differential Geometry School held at Manaus, AM. He wishes to thank Professor Renato Tribuzy and all the other members at the Department of Mathematics for their hospitality. The authors wish also to thank Professor Wies{\l}aw \'{S}liwa for kindly sending us a copy of his paper \cite{Sliwa}. This work is partially supported by CNPq-Brazil, Grant 307210/2009-0.

\vskip1truecm

\scshape

\noindent Cleon S. Barroso and Michel P. Rebou\c cas

\noindent Universidade Federal do Cear\'a

\noindent Departmento de Matem\'atica

\noindent Av. Humberto Monte S/N, 60455-760, Bl 914

\noindent E-mail address: \qquad {\tt cleonbar@mat.ufc.br}\quad and \quad {\tt michelufc@yahoo.com.br}

\scshape
\bigskip

\noindent Marcus A. M. Marrocos

\noindent Universidade Federal do Amazonas

\noindent Departamento de Matem\'atica, ICE

\noindent Av. Rodrigo Ot\'avio Jord\~ao Ramos, 3000, 3077-000

\noindent E-mail address: \qquad {\tt marcusmarrocos@gmail.com}


\begin{thebibliography}{00}


\bibitem{ADK} S. A. Argyros, P. Dodos and V. Kanellopoulos, Unconditional families in Banach spaces, Math. Annalen, {\bf 341} (2008), 15--38.

\bibitem{Aron1} R. M. Aron, V. I. Gurariy and J. B. Seoane-Sep\'ulveda, Lineability and spaceability of sets of functions on $\mathbb{R}$, Proc. Amer. Math. Soc. {\bf 133} (2005), 795--803.

\bibitem{Aron2} R. M. Aron, F. J. Grac\'\i a-Pacheco, D. P\'erez-Garc\'\i a and J. B. Seoane-Sep\'ulved, On dense-lineability of sets of functions on $\mathbb{R}$, Topology {\bf 48} (2009), 149--156.

\bibitem{Astala} K. Astala, On Peano's theorem in locally convex spaces, Studia Math. {\bf 73} (1982), 213--223.


\bibitem{Barroso} C. S. Barroso, The approximate fixed point property in Hausdorff topological vector spaces and applications, Discrete Contin. Dyn. Syst. 25 (2009), no. 2, 467–-479.

\bibitem{Barroso2} C. S. Barroso, G. Botelho, V. V. F\'avaro and D. Pellegrino, Lineability and spaceability for the weak form of Peano's theorem and vector-valued sequence spaces. To appear in Proc. AMS.

\bibitem{Barroso-Kalenda-Lin} C. S. Barroso, O. F. K. Kalenda and P.-K. Lin, On the weak approximate fixed point property in abstract spaces, Math. Z. {\bf 271} (2012), no. 3-4, 1271–-1285.


\bibitem{Carreras-Bonet} P. P\'erez-Carreras and J. Bonet, Barreled Locally Convex Spaces. North-Holland Mathematics Studies, 131. Notas de Matemática [Mathematical Notes], 113. North-Holland Publishing Co., Amsterdam, 1987.

\bibitem{Cellina} A. Cellina, On the nonexistence of solutions of differential equations in nonreflexive spaces. Bull. Amer. Math. Soc. {\bf 78} (1972), 1069–-1072.

\bibitem{Cembranos} P. Cembranos, Algunas propriedades del espacio de Banach $C(K,X)$, Ph. D. Thesis, Universidad Complutense de Madrid, Madrid 1984.


\bibitem{Dieudonne} J. Dieudonn\'e, Deux examples singuliers d\'equations differentielles, Acta Sci. Math. (Szeged) {\bf 12B} (1950), 38–-40.

\bibitem{Fabian} M. Fabian, P. Habala, P. H\'ajek, V. Montesinos, J. Pelant, and V. Zizler, Functional Analysis and Infinite Dimensional
Geometry, CMS Books in Mathematics {\bf 8}, Springer-Verlag, 2001.


\bibitem{Godunov} A. N. Godunov, The Peano theorem in Banach spaces. (Russian) Funkcional. Anal. i Prilo\u{z}en. {\bf 9} (1974), 59–-60.

\bibitem{HJ} J. N. Hagler and W. B. Johnson, On Banach spaces whose dual balls are not weak* sequentially compact, Israel J. Math., {\bf 28} (1977), 325--330.


\bibitem{Hajek-Johanis} P. H\'ajek and M. Johanis, On Peano's theorem in Banach space, J. Diff. Equations., {\bf 249} (2010), 3342--3351.


\bibitem{JR} W. B. Johnson and H. P. Rosenthal, On weak* basic sequences and their applications to the study of Banach spaces, Studia Math., {\bf 43} (1972), 77--92.


\bibitem{Lobanov} S. G. Lobanov, On Peano's theorem in Fr\'echet spaces, Differentsial'nye Uravneniya {\bf 28} (1992), 1086.

\bibitem{Peano} G. Peano, Opere Scelte, vol. 1 (Edizioni Cremonese, 1957).

\bibitem{Pe} A. Pe{\l}czy\'nski, On Banach spaces containing $L_1(\mu)$, Studia Math. {\bf 30} (1968), 231--246.

\bibitem{Mujica} J. Mujia, Separable quotients of Banach spaces, Rev. Math. {\bf 10} (1997), 299--330.

\bibitem{Rosenthal} H. P. Rosenthal, A characterization of Banach spaces containing $\ell_1$, Proc. Nat. Acad. Sci. U.S.A. {\bf 71} (1974), 2411–-2413.

\bibitem{Schluchtermann} G. Schl\"{u}chtermann, Weak compactness in $L_\infty(\mu,X)$, J. Funct. Anal. {\bf 125} (1994), 379-–388.

\bibitem{Shkarin1} S. A. Shkarin, On a problem of O. G. Smolyanov related to Peano's infinte-dimensional theorem, Differentsial'nye Uravneniya {\bf 28} (1992), 1092.

\bibitem{Shkarin2} S. Shkarin, On Osgood theorem in Banach spaces, Math. Nahchr. {\bf 257} (2003), 87--98.

\bibitem{Sliwa} W. \'{S}liwa, The separable quotient problem and the strongly normal sequences, J. Math. Soc. Japan {\bf 64} (2012), 387--397.

\bibitem{Szep} A. Sz\'ep, Existence theorem for weak solution of ordinary differential equations in reflexive Banach spaces, Stud. Sci. Math. Hung. {\bf 6} (1971), 197--203.

\bibitem{Teixeira} E. V. Teixeira, Strong solutions for differential equations in abstract spaces, J. Differential Equations {\bf 214} (2005), 65–-91.

\end{thebibliography}
\end{document}